\documentclass[a4paper, 10pt]{article}\usepackage{amsmath,amssymb,exscale,amsthm}\usepackage{a4wide}\usepackage[english]{babel}

\begin{document}
\title{On the existence of $(H,A)$-stable sheaves on K3 or abelian surfaces}
\author{Markus Zowislok\footnote{Department of Mathematics, Imperial College London, 180 Queen's Gate, London SW7 2AZ, UK}}

\date{}
\maketitle

\newcommand{\IQ}{\ensuremath{\mathbb{Q}}}\newcommand{\IR}{\ensuremath{\mathbb{R}}}\newcommand{\IZ}{\ensuremath{\mathbb{Z}}}
\newcommand{\IP}{\ensuremath{\mathbb{P}}}\newcommand{\IC}{\ensuremath{\mathbb{C}}}\newcommand{\IN}{\ensuremath{\mathbb{N}}}
\newcommand{\cO}{\ensuremath{\mathcal{O}}} 
\newcommand{\bc}{\ensuremath{\mathrm{c}}}\newcommand{\bH}{\ensuremath{\mathrm{H}}}
\newcommand{\ch}{\ensuremath{\mathrm{ch}}}\newcommand{\td}{\ensuremath{\mathrm{td}}}
\newcommand{\ext}{\ensuremath{\mathrm{ext}}}\newcommand{\Ext}{\ensuremath{\mathrm{Ext}}} 
\renewcommand{\hom}{\ensuremath{\mathrm{hom}}}\newcommand{\Hom}{\ensuremath{\mathrm{Hom}}}
\newcommand{\rk}{\ensuremath{\mathrm{rk\,}}}\newcommand{\codim}{\ensuremath{\mathrm{codim}}}
\newcommand{\Coh}{\ensuremath{\mathrm{Coh}}}
\newcommand{\Pic}{\ensuremath{\mathrm{Pic}}}\newcommand{\Num}{\ensuremath{\mathrm{Num}}}
\newcommand{\CH}{\ensuremath{\mathrm{CH}}}\newcommand{\NS}{\ensuremath{\mathrm{NS}}}
\newcommand{\Amp}{\ensuremath{\mathrm{Amp}}}
\newcommand{\lel}{\ensuremath{\;(\le)\;}} \newcommand{\geg}{\ensuremath{\;(\ge)\;}} 
\newcommand{\qede}{\vspace{-1.3cm}\[\qedhere\]}
\newcommand{\qedel}{\vspace{-1.0cm}\[\qedhere\]}
\renewcommand{\labelenumi}{(\arabic{enumi})\ }\renewcommand{\labelenumii}{(\alph{enumii})\ }

\newtheorem{Prop}{Proposition}[section]
\newtheorem{Def}[Prop]{Definition}
\newtheorem{Cor}[Prop]{Corollary}
\newtheorem{Th}[Prop]{Theorem}
\newtheorem{Lemma}[Prop]{Lemma}

\selectlanguage{english}

\begin{abstract}
We give an existence result on $(H,A)$-stable sheaves on a K3 or abelian surface $X$ with primitive triple of invariants (rank,first Chern class,Euler characteristics) in the integral cohomology lattice.
Such a result yields the existence of singular projective $\IQ$-factorial symplectic terminalisations of certain moduli spaces of sheaves on $X$ that are Gieseker semistable with respect to a nongeneral ample divisor.
\end{abstract}

\section{Introduction}

After the paper \cite{KLS06} has appeared, the hope to construct new examples of irreducible (holomorphically) symplectic manifolds out of moduli spaces of sheaves on K3 or abelian surfaces almost died: the authors showed that in general, i.e.\ for general ample divisors, there is no symplectic resolution of these moduli spaces except for the nonsingular and O'Grady-like cases.
In \cite{Zow12} I investigated the case of a nongeneral ample divisor.
In particular, I could exclude the existence of new examples of projective irreducible symplectic manifolds lying birationally over components of the moduli spaces of one-dimensional semistable sheaves on K3 surfaces, and over components of many of the moduli spaces of two-dimensional sheaves on K3 surfaces, in particular, of those for rank two sheaves.

In order to answer the question of symplectic resolvability, as explained in \cite{Zow12}, constructing a projective $\IQ$-factorial symplectic terminalisation $\tilde M\to M$ of a component $M$ of the moduli space, i.e.\ a symplectic $\IQ$-factorial projective variety $\tilde M$ with at most terminal singularities together with a projective birational morphism $f\colon \tilde M\to M$, yields the following facts:
\begin{enumerate}
\item If $\tilde M$ can be chosen to be an irreducible symplectic manifold then $\tilde M$ is unique up to deformation by a result of Huybrechts \cite{Huy99}.
\item If $\tilde M$ is singular, $M$ admits no projective symplectic resolution by \cite[Corollary 1]{Nam06}.
\end{enumerate}

To be more precise we need some notation.
Let $X$ be a nonsingular projective irreducible surface over $\IC$, $K_X$ its canonical divisor, $H$ an ample divisor on $X$, and $E$ a coherent sheaf on $X$. We associate the element
$$u(E):=(\rk E,\bc_1(E),\chi(E))\in\Lambda(X):=\IN_0\oplus\NS(X)\oplus\IZ\subset H^{2*}(X,\IZ)$$ of sheaf invariants to $E$.
We avoid the elegant notion of a Mukai vector in favour of keeping torsion inside $\NS(X)$.
For an element $u:=(r,c,\chi)\in \Lambda(X)$ we define
\begin{eqnarray*}
%P(u)&:=&r \frac{H^2}2n^2+\left(c-r\;\frac{K_X}2\right).Hn+\chi\,,\\
\Delta(u)&:=&c^2 - 2r\chi+2r^2\chi(\cO_X)-rc.K_X\quad\textrm{and}\\
\chi(u,u)&:=&\chi(\cO_X) r^2-\Delta(u)\,.
\end{eqnarray*}
If $E$ satisfies $u(E)=u$, then, by Riemann-Roch, %its Hilbert polynomial is $P(u)$, 
its discriminant\footnote{Be aware of different conventions of the discriminant's definition.} is $\Delta(u)$, and
$$\chi(E,E):=\sum_{k=0}^2 \ext^k(E,E) = \chi(u,u)\,,$$ where $\ext^k(E,E):=\dim \Ext^k(E,E)$. We will also write $\hom(E,F):=\dim\Hom(E,F)$ for two coherent sheaves $E,F$.
We denote the moduli space of sheaves $E$ on $X$ with $u(E)=u$ that are semistable with respect to an ample divisor $H$ on $X$ by $M_H(u)$ and the open subscheme of stable sheaves by $M^s_H(u)$. The corresponding spaces for $(H,A)$-(semi)stable sheaves introduced in \cite{Zow12} are denoted by $M_{H,A}(u)$ and $M^s_{H,A}(u)$.

The main result of \cite{Zow12} on the case of positive rank was the extension of the result of \cite{KLS06} to the following\\

{\noindent \bf Theorem \cite{Zow12} 1.1.\ \it
Let $X$ be a projective K3 or abelian surface, $u=(r,c,\chi)\in\Lambda(X)$ primitive with $r>0$ and $\chi(u,u)\ge 0$, $m\in\IN$ and $H$ an ample divisor on $X$, and assume that $M^s_H(mu)$ is nonempty. 
\begin{enumerate}
\item Let $m=1$ or $\chi(mu,mu)=8$.
Then there is a projective symplectic resolution
$M\to \overline{M^s_H(mu)}$.
If $H$ is not $mu$-general then $M$ can be chosen to be a symplectic resolution of $M_{H,A}(mu)$, where $A$ is an $mu$-general ample divisor.
\item Let $m\ge 2$ and $\chi(mu,mu)\neq 8$.
If $H$ is $mu$-general or $r=1$ or $\chi(u,u)>\varphi(r)$ with $\varphi$ as in \cite[Theorem 6.5]{Zow12}
then there is a singular $\IQ$-factorial projective symplectic terminalisation 
of $\overline{M^s_H(mu)}\,,$
and in particular, there is no projective symplectic resolution of $\overline{M^s_H(mu)}$.
\end{enumerate}
}

\noindent The proof of (2) is based on the existence of a singular $\IQ$-factorial projective symplectic terminalisation $M\to M_{H,A}(mu)$ established by item 2.b.ii of \cite[Theorem 5.3]{Zow12} using the existence of an $(H,A)$-stable sheaf $E$ with $u(E)=u$.
This existence is ensured by the assumption of (2), see the proof of the above theorem in \cite{Zow12}. Of course, instead one can also just assume this existence. Our main result of this article is another existence result, which in turn implies the existence of a singular $\IQ$-factorial projective symplectic terminalisation as in the above theorem:\\

{\noindent \bf Theorem \ref{neindep}.\ \it
Let $X$ be a projective surface with torsion canonical bundle, $u\in\Lambda(X)$ primitive, and $H$ and $A$ two ample divisors on $X$ such that $H$ is contained in at most one wall and $A$ is $u$-general.
Then the nonemptyness of $M_{H,A}(u)$ is independent of the choice of the pair $(H,A)$.
}\\

\noindent In particular, one has:\\

{\noindent \bf Corollary \ref{MHAne}.\ \it
Let $X$ be a projective K3 or abelian surface, $u\in\Lambda(X)$ primitive with $\chi(u,u)\ge -2$, and $H$ and $A$ two ample divisors on $X$ such that $H$ is contained in at most one wall and $A$ is $u$-general.
Then $M_{H,A}(u)$ is nonempty.
}\\

\noindent As $M_{H,A}(u)=M^s_{H,A}(u)$, in the situation of the corollary there is an $(H,A)$-stable sheaf $E$ with $u(E)=u$.

\vspace{1cm}
{\small \noindent \textit{Acknowledgements.}
The author would like to express his gratitude to Richard Thomas for his support and valuable discussions.
Moreover, he thanks the Imperial College London for its hospitality and the Deutsche Forschungsgemeinschaft (DFG) for supporting the stay there by a DFG research fellowship (Az.: ZO 324/1-1).}

\section{Twisted and $(H,A)$-stability}

In this section we recall three notions of stability of sheaves and establish a relation between twisted stability and $(H,A)$-stability for positive rank.
In my PhD thesis \cite{Zow10}, this relation was discussed in Chapter 6 for K3 surfaces.
We assume familiarity with the material presented in \cite{HL10} and use the notation therein.

Let still $X$ be a nonsingular projective irreducible surface over $\IC$. In this case, twisted stability and $(H,A)$-stability, which are two generalisations of Gieseker stability, have an overlap. We briefly recall the definitions. Therefore let $H$ be an ample divisor on $X$ and $E$ a nontrivial coherent sheaf on $X$.

\begin{enumerate}
\item \textit{Gieseker stability, see e.g.\ in \cite[Section 1.2]{HL10}.} 
The Hilbert polynomial of $E$ is $P_H(E)(n):=\chi(E\otimes \cO_X(nH))$. Its leading coefficient multiplied by $(\dim E)!$ is called multiplicity of $E$ and denoted here by $\alpha^H(E)$. It is always positive, and $$p_H(E)(n):=\frac{\chi(E\otimes \cO_X(nH))}{\alpha^H(E)}$$ is called reduced Hilbert polynomial of $E$.
$E$ is said to be $H$-(semi)stable if $E$ is pure and for all nontrivial proper subsheaves $F\subset E$ one has that $p_H(F) \lel p_H(E)$, i.e.\ one has $p_H(F)(n) \lel p_H(E)(n)$ for $n\gg 0$. 

In order to avoid case differentiation for stable and semistable sheaves we here follow the Notation 1.2.5 in \cite{HL10} using bracketed inequality signs, e.g.\ an inequality with $\lel$ for (semi)stable sheaves means that one has $\le$ for semistable sheaves and $<$ for stable sheaves.

If $\rk E>0$, then $E$ is $H$-(semi)stable if $E$ is pure and for all nontrivial proper subsheaves $F\subset E$ one has that $\mu_H(F) \le \mu_H(E)$ and, in the case of equality, $\frac{\chi(F)}{\rk F} \lel \frac{\chi(E)}{\rk E}$.
Here $\mu_H(E):=\frac{\bc_1(E).H}{\rk E}$ is the slope of $E$ (with respect to $H$).

\item \textit{Twisted stability.}
Let $D\in\NS(X)_\IQ:=\NS(X)\otimes\IQ$.
We call $$\chi^D(E):=\int_X\ch(E).\exp(D).\td(X)$$ the $D$-twisted Euler characteristic of $E$, and
we say that $E$ is $D$-twisted $H$-(semi)stable if $E$ is pure and for all nontrivial saturated proper subsheaves $F\subset E$ one has that $$\frac{\chi^{D+nH}(F)}{\alpha^H(F)}\lel \frac{\chi^{D+nH}(E)}{\alpha^H(E)}$$
as polynomials in $n$.

If $\rk E>0$, then $E$ is $D$-twisted $H$-(semi)stable if $E$ is pure and for all nontrivial proper subsheaves $F\subset E$ one has that $\mu_H(F) \le \mu_H(E)$ and, in the case of equality, $$\mu_D(F)+\frac{\chi(F)}{\rk F} \lel \mu_D(E)+\frac{\chi(E)}{\rk E}\,.$$

\item \textit{$(H,A)$-stability as defined in \cite[Definition 7.1]{Zow12}.}
We only give an equivalent definition for sheaves of positive rank on a surface.
Let $A$ be another ample divisor on $X$ and assume that $\rk E>0$.
Then $E$ is $(H,A)$-(semi)stable if it is $H$-semistable and if for any proper nontrivial subsheaf $F\subset E$ with reduced Hilbert polynomial $p_H(F)=p_H(E)$ one has that $\mu_A(F) \geg \mu_A(E)$, i.e.\ stable corresponds to $>$ and semistable to $\ge$.

It is enough to restrict to saturated proper nontrivial subsheaves $F\subset E$ in the definition.
\end{enumerate}
The case of Gieseker stability can be regained by $D=0$ from twisted stability and by $H=A$ from $(H,A)$-stability.\\

We briefly recall the notion of a general ample divisor for positive rank.
The ample cone of $X$ carries a chamber structure for a given triple $u=(r,c,\chi)\in\Lambda(X)$ of invariants.
The definition depends on $r$. In the case of $r=1$ we agree that the whole ample cone is the only chamber.
For $r>1$, we follow the definition in \cite[Section 4.C]{HL10}.
Let $\Num(X):=\Pic(X)/\equiv$, where $\equiv$ denotes numerical equivalence, and $\Delta:=\Delta(u)>0$.
\begin{Def}
Let
$$W(r,\Delta):=\{ \xi^\perp \cap \Amp(X)_{\IQ} \;|\; \xi\in\Num(X) \quad\mathrm{with}\quad -\frac {r^2}4 \Delta \le \xi^2 < 0 \}\,,$$
whose elements are called $u$-walls.
The connected components of the complement of the union of all $u$-walls are called $u$-chambers.
An ample divisor is called $u$-general if it is not contained in a $u$-wall.
\end{Def}
\noindent The set $W(r,\Delta)$ is locally finite in $\Amp(X)_\IQ$ by \cite[Lemma 4.C.2]{HL10}.\\

Let still $r>0$, $H$ an ample divisor lying on exactly one $u$-wall $W$ and $A$ a $u$-general ample divisor lying in a chamber touching $H$.

\begin{Def}
For a nontrivial saturated subsheaf $F\subset E$ of a $\mu_H$-semistable sheaf $E$ with $u(E)=u$, $\mu_H(F)=\mu_H(E)$, and $$\frac{\bc_1(F)}{\rk F}\not\equiv\frac{\bc_1(E)}{\rk E}\,,$$ we call the hyperplane
$$\left\{z\in \NS(X)_\IQ \;\Big|\; \frac{\chi^z(F)}{\rk F}=\frac{\chi^z(E)}{\rk E}  \right\}$$
a $u$-miniwall.
The connected components of the complement of all $u$-miniwalls are called $u$-minichambers.\footnote{Both notions are inspired by the work of Ellingsrud and G\"ottsche.}
\end{Def}
In the following we omit the $u$-prefix as it is fixed for the whole section.

\begin{Prop}
The number of miniwalls is finite and the miniwalls are parallel to $W$.
For $D,D'\in \NS(X)_\IQ$ one has that the set of $D$-twisted $H$-semistable sheaves is the same as the set of $D'$-twisted $H$-semistable sheaves
if and only if $D$ and $D'$ belong to the same $v$-minichamber or $v$-miniwall.
\end{Prop}
\begin{proof}
\cite[Proposition 3.5]{MW97}.
\end{proof}

\begin{Lemma}
Let $D$ be contained in a minichamber and $E$ a $D$-twisted $H$-semistable sheaf with $u(E)=u$. Then for every nontrivial saturated subsheaf $F\subset E$ with $$\frac{\chi^{D+nH}(F)}{\rk F}= \frac{\chi^{D+nH}(E)}{\rk E}$$ (as polynomials in $n$) one has that $$\frac{\bc_1(F)}{\rk F}\equiv\frac{\bc_1(E)}{\rk E}\quad\textit{and}\quad \frac{\chi(F)}{\rk F}=\frac{\chi(E)}{\rk E}\,.$$
\end{Lemma}
\begin{proof}
Let $F\subset E$ be such a nontrivial saturated subsheaf. Equating the coefficients of the above polynomials yields $\mu_H(F)=\mu_H(E)$ and $$\frac{\chi^D(F)}{\rk F}=\frac{\chi^D(E)}{\rk E}\,.$$
As $D$ is not contained in a miniwall, one has that $\frac{\bc_1(F)}{\rk F}\equiv\frac{\bc_1(E)}{\rk E}$
and thus also $\frac{\chi(F)}{\rk F}=\frac{\chi(E)}{\rk E}\,.$
\end{proof}

\begin{Lemma}\label{tStab}
Let $L$ be in a minichamber $C$, $L'$ in its closure $\overline C$, and $E$ a coherent sheaf on $X$ with $u(E)=u$.
\begin{enumerate}
\item If $E$ is $L$-twisted  $H$-semistable then it is also $L'$-twisted $H$-semistable.
\item If $E$ is $L'$-twisted $H$-stable then it is also $L$-twisted $H$-stable.
\end{enumerate}
\end{Lemma}
\begin{proof}
Let $F\subset E$ be a nontrivial saturated proper subsheaf.
As for $\mu_H(F)<\mu_H(E)$ one has $$\frac{\chi^{nH+D}(F)}{\rk F}<\frac{\chi^{nH+D}(E)}{\rk E}$$ (as polynomials in $n$) for any $D\in \NS(X)_\IQ$, we can restrict to $\mu_H(F)=\mu_H(E)$. We define the map
$$f\colon \overline C\to \IQ, D\mapsto \left(\frac{\bc_1(F)}{\rk F}-\frac{\bc_1(E)}{\rk E}\right).D+\frac{\chi(F)}{\rk F}-\frac{\chi(E)}{\rk E}\,.$$
If $\frac{\bc_1(F)}{\rk F}\equiv\frac{\bc_1(E)}{\rk E}$ then $f$ is independent of $D$. So let $\frac{\bc_1(F)}{\rk F}\not\equiv\frac{\bc_1(E)}{\rk E}$. Then $f\neq 0$ on the whole minichamber $C$ by the definition of a minichamber.
We distinguish the two cases from above.
\begin{enumerate}
\item Let $E$ be $L$-twisted $H$-semistable. Then $f<0$ on $C$, hence $f\le 0$.
\item Let $E$ be $L'$-twisted $H$-stable. Then $f(L')<0$, hence $f<0$ on an open subset containing $L'$, which in turn yields $f<0$ on $C$.\qedhere
\end{enumerate}
\end{proof}

\begin{Prop}\label{tSttAHSt}
Let $L$ be in a minichamber $C$, $L'$ in its boundary $\partial C$, and $E$ a coherent sheaf on $X$ with $u(E)=u$.
The vector space generated by the wall $W$ divides $\NS(X)_\IQ$ into two open half spaces, one of them containing $L-L'$.
Choose $A$ in the neighbouring chamber of $W$ contained in the other half space.
Then $E$ is $L$-twisted $H$-(semi)stable if and only if it is $L'$-twisted $H$-semistable and for all nontrivial saturated proper subsheaves $F\subset E$ with $$\frac{\chi^{nH+L'}(F)}{\rk F}= \frac{\chi^{nH+L'}(E)}{\rk E}$$ (as polynomials in $n$) one has that $\mu_A(F)\geg\mu_A(E)$.
\end{Prop}
\begin{proof}
Let $F\subset E$ be a nontrivial saturated proper subsheaf. 
As for $\mu_H(F)<\mu_H(E)$ one has $$\frac{\chi^{nH+D}(F)}{\rk F}<\frac{\chi^{nH+D}(E)}{\rk E}$$ (as polynomials in $n$) for any $D\in \NS(X)_\IQ$, we can again restrict to $\mu_H(F)=\mu_H(E)$. Then
\begin{eqnarray}
\left(\frac{\chi^{nH+L}(F)}{\rk F}-\frac{\chi^{nH+L}(E)}{\rk E}\right)-\left(\frac{\chi^{nH+L'}(F)}{\rk F}-\frac{\chi^{nH+L'}(E)}{\rk E}\right)=\left(\frac{\bc_1(F)}{\rk F}-\frac{\bc_1(E)}{\rk E}\right).(L-L')\,.\label{rtHPDD}
\end{eqnarray}
If $\frac{\bc_1(E)}{\rk E}\equiv\frac{\bc_1(F)}{\rk F}$ then 
$$\frac{\chi^{nH+L}(F)}{\rk F}-\frac{\chi^{nH+L}(E)}{\rk E}=\frac{\chi^{nH+L'}(F)}{\rk F}-\frac{\chi^{nH+L'}(E)}{\rk E}$$
and $\mu_A(F)=\mu_A(E)$,
so we assume $$\frac{\bc_1(F)}{\rk F}-\frac{\bc_1(E)}{\rk E}\not\equiv 0\,,$$
which thus defines the wall $W$. In particular, the sign of 
$$\left(\frac{\bc_1(F)}{\rk F}-\frac{\bc_1(E)}{\rk E}\right).(L-L')\neq 0$$
is opposite to the sign of $\mu_A(F)-\mu_A(E)$ due to the choice of $A$.
\begin{enumerate}
\item
Assume that $E$ is $L$-twisted $H$-semistable and thus also $L'$-twisted $H$-semistable by Lemma \ref{tStab}.
If furthermore $$\frac{\chi^{nH+L'}(F)}{\rk F}=\frac{\chi^{nH+L'}(E)}{\rk E}$$ then equation (\ref{rtHPDD}) yields
$$\frac{\chi^{nH+L}(F)}{\rk F}-\frac{\chi^{nH+L}(E)}{\rk E}=\left(\frac{\bc_1(F)}{\rk F}-\frac{\bc_1(E)}{\rk E}\right).(L-L')\,,$$
which is negative, hence $\mu_A(F)>\mu_A(E)$.
\item 
Assume that $E$ is $L'$-twisted $H$-semistable, i.e.\ in particular $$\frac{\chi^{nH+L'}(F)}{\rk F}\le\frac{\chi^{nH+L'}(E)}{\rk E}\,.$$
If one has strict inequality then by the same argument as in Lemma \ref{tStab} one has that $$\frac{\chi^{nH+L}(F)}{\rk F}<\frac{\chi^{nH+L}(E)}{\rk E}\,.$$
So let's assume equality. Then $\mu_A(F)\ge\mu_A(E)$
and thus
$$\frac{\chi^{nH+L}(F)}{\rk F}-\frac{\chi^{nH+L}(E)}{\rk E}=\left(\frac{\bc_1(F)}{\rk F}-\frac{\bc_1(E)}{\rk E}\right).(L-L')<0\,.$$
\end{enumerate} \qedel
\end{proof}

\noindent The following statement, at least the part on semistability, is already known to Matsuki and Wentworth, as it can be found in \cite[Theorem 4.1, part i]{MW97}.

\begin{Cor}\label{tStAHSt}
Let $A$ be an ample divisor in a chamber touching $H$ and $L\in\Pic(X)$ lying on a miniwall.
The vector space generated by the wall $W$ divides $\NS(X)_\IQ$ into two open half spaces, one of them containing $A$.
Choose $D$ in one of the minichambers touching $L$ such that $D-L$ is in the other half space.
Then a coherent sheaf $E$ with $u(E)=u$ is $D$-twisted $H$-(semi)stable if and only if $E\otimes L$ is $(H,A)$-(semi)stable.
\end{Cor}
\begin{proof}
Clearly a coherent sheaf $E$ is $L$-twisted $H$-(semi)stable if and only if $E\otimes L$ is $H$-(semi)stable.
Thus the claim follows from Proposition \ref{tSttAHSt} and the description of $(H,A)$-stability at the beginning of this section.
\end{proof}

\section{Existence of $(H,A)$-stable sheaves}

\begin{Th}\label{neindep}
Let $X$ be a projective surface with torsion canonical bundle, $u\in\Lambda(X)$ primitive, and $H$ and $A$ two ample divisors on $X$ such that $H$ is contained in at most one wall and $A$ is $u$-general.
Then the nonemptyness of $M_{H,A}(u)$ is independent of the choice of the pair $(H,A)$.
\end{Th}
\begin{proof}
As a direct consequence of \cite[Proposition 4.1]{Yos03}, the nonemptyness of the moduli space $M_H^D(u)$ of $D$-twisted $H$-stable sheaves is independent of the choice of the pair $(H,D)$ if $(H,D)$ is $u$-general, where $D$ is any $\IQ$-line bundle. 
The claim now follows from Corollary \ref{tStAHSt}.
\end{proof}

\noindent Hence it is enough to prove nonemptyness for one suitable special choice of ample divisors. In particular, one has the

\begin{Cor}\label{MHAne}
Let $X$ be a projective K3 or abelian surface, $u\in\Lambda(X)$ primitive with $\chi(u,u)\ge -2$, and $H$ and $A$ two ample divisors on $X$ such that $H$ is contained in at most one wall and $A$ is $u$-general.
Then $M_{H,A}(u)$ is nonempty.
\end{Cor}
\begin{proof}
This follows from the above Theorem \ref{neindep} as $M_H(u)=M_{H,H}(u)$ is well-known to be nonempty for general $H$ and $\chi(u,u)\ge -2$, see e.g.\ \cite{KLS06}.
\end{proof}

\bibliographystyle{amsalpha}
\bibliography{my}

\providecommand{\bysame}{\leavevmode\hbox to3em{\hrulefill}\thinspace}
\providecommand{\MR}{\relax\ifhmode\unskip\space\fi MR }
\providecommand{\MRhref}[2]{%
  \href{http://www.ams.org/mathscinet-getitem?mr=#1}{#2}
}
\providecommand{\href}[2]{#2}
\begin{thebibliography}{Nam06}

\bibitem[HL10]{HL10}
Daniel Huybrechts and Manfred Lehn, \emph{The geometry of moduli spaces of
  sheaves}, second ed., Cambridge Mathematical Library, Cambridge University
  Press, Cambridge, 2010.

\bibitem[Huy99]{Huy99}
Daniel Huybrechts, \emph{Compact hyper-{K}{\"a}hler manifolds: basic results},
  Invent. Math. \textbf{135} (1999), no.~1, 63--113.

\bibitem[KLS06]{KLS06}
D.~Kaledin, M.~Lehn, and Ch. Sorger, \emph{Singular symplectic moduli spaces},
  Invent. Math. \textbf{164} (2006), no.~3, 591--614.

\bibitem[MW97]{MW97}
Kenji Matsuki and Richard Wentworth, \emph{Mumford-{T}haddeus principle on the
  moduli space of vector bundles on an algebraic surface}, Internat. J. Math.
  \textbf{8} (1997), no.~1, 97--148.

\bibitem[Nam06]{Nam06}
Yoshinori Namikawa, \emph{On deformations of {$\mathbb Q$}-factorial symplectic
  varieties}, J. Reine Angew. Math. \textbf{599} (2006), 97--110.

\bibitem[Yos03]{Yos03}
K{\=o}ta Yoshioka, \emph{Twisted stability and {F}ourier-{M}ukai transform.
  {I}}, Compositio Math. \textbf{138} (2003), no.~3, 261--288.

\bibitem[Zow10]{Zow10}
Markus Zowislok, \emph{On moduli spaces of semistable sheaves on {K3}
  surfaces}, Dissertation, S{\"u}dwestdeutscher Verlag f{\"u}r
  Hochschulschriften, 2010, http://ubm.opus.hbz-nrw.de/volltexte/2010/2287/.

\bibitem[Zow12]{Zow12}
\bysame, \emph{On moduli spaces of sheaves on {K}3 or abelian surfaces}, Math.
  Z. \textbf{272} (2012), no.~3-4, 1195--1217.

\end{thebibliography}

\end{document}